\documentclass[12pt]{amsart}
\usepackage{amsfonts,amsmath,amssymb,amscd,amssymb}
\usepackage{enumerate}
\usepackage{enumitem}
\usepackage{comment}
\usepackage{hyperref}
\newtheorem{theorem}{Theorem}[section]
\newtheorem{lemma}[theorem]{Lemma}

\newcommand{\del}{\delta}

\newcommand{\Q}{\mbox{$\mathbb Q$}}
\newcommand{\R}{\mbox{$\mathbb R$}}     % For Real numbers
\begin{document}
	\title[]{Prime ideal divisors of parametric recurrence sequences
    }
%    \begin{comment}
	\author[Darsana]{Darsana N}
	\address{Darsana N, Department of Mathematics, National Institute of Technology Calicut, 
		Kozhikode-673 601, India.}
	\email{darsana\_p230059ma@nitc.ac.in; darsanasri1996@gmail.com}

	\author[Rout]{S. S. Rout}
	\address{Sudhansu Sekhar Rout, Department of Mathematics, National Institute of Technology Calicut, 
		Kozhikode-673 601, India.}
	\email{sudhansu@nitc.ac.in; lbs.sudhansu@gmail.com}
%\end{comment}

	\dedicatory{}
    
	\maketitle
	\pagenumbering{arabic}
	\pagestyle{headings}
    \begin{abstract}
We prove new arithmetic results for parametric linear recurrence sequences specialized at roots of unity, denoted by $(U_n(\zeta))_{n\geq 0}$. In particular, we obtain effective lower bounds for the largest prime ideal divisor and norm of the radical of the principal ideal generated by $U_n(\zeta)$. We further derive an effective upper bound for the $S$-part of $U_n(\zeta)$, showing that it is strictly smaller than a fixed power of its absolute norm for sufficiently large $n$.
\end{abstract}

    \section{Introduction}
    We recall that a linear recurrence sequence $(u_n)_{n\geq 0}$
 of order $k$ over a number field $K$ is a sequence which satisfies a relation of the form
\[
u_{n+k}=A_{k-1}u_{n+k-1}+\cdots+A_0u_n, \quad n\geq 0,
\] with some constants $A_0,\ldots, A_{k-1}, u_0,\ldots,u_{k-1}\in K$. Let $\alpha_1,\dots,\alpha_q$ be the distinct roots of the corresponding characteristic polynomial  
\begin{equation}\label{eq5}
	F(X):= X^k - A_{k-1}X^{k-1}-\dots-A_0.
\end{equation} 
Then for $n\geq 0$, we have 
\begin{equation}\label{eq6}
	u_n=f_1(n)\alpha_1^n +\cdots + f_q(n)\alpha_{q}^n,
\end{equation}
where $f_i(n)\; (i=1, \ldots, q)$ are non-zero polynomials with degree less than the multiplicity of $\alpha_i$; the coefficients of $f_i(n)$ are elements of the number field $\Q( \alpha_1, \ldots, \alpha_q)$ (see \cite[Theorem C.1]{tnshorey}).

The sequence $(u_n)_{n \geq 0}$ is called {\it degenerate} if there are integers $i, j$ with $1\leq i< j\leq q$ such that $\alpha_i/\alpha_j$ is a root of unity; otherwise it is called {\it non-degenerate}. The sequence $(u_n)_{n \geq 0}$ is called {\it simple} if $q=k$. In this case, \eqref{eq6} becomes
\begin{equation}\label{eqa}
	u_n=f_1\alpha_1^n +\cdots + f_k\alpha_{k}^n,
\end{equation}
where $f_i \; (i=1, \ldots, k)$ are constants.

For any integer $m$, let $P(m)$ denote the greatest prime factor of $m$ and
$Q(m)$ denote the greatest square-free factor of $m$ with the convention that $P(0) = P(\pm 1) = 1$ and $ Q(0)= Q(\pm 1) =1$. That is, if $m = p_1^{h_1}\ldots p_r^{h_r}$ with $p_1<\cdots<p_r$
 primes and $h_1,\ldots,h_r$ positive integers, then $P(m)=p_r$ and  $Q(m) = p_1\ldots p_r$. 
In 1935, Mahler~\cite{mahler1934arithmetische} showed that for any
non-degenerate linear recurrence sequence $(u_n)_{n\geq 0}$, the absolute value of
its terms grows without bound; more precisely,
\begin{equation}\label{ieq1}
|u_n| \longrightarrow \infty \quad \text{as } n \longrightarrow \infty .
\end{equation}
Later, van der Poorten and Schlickewei~\cite{vanderschlickewei1982}, as well
as Evertse~\cite{evertse1984sums}, using a $p$-adic version of Schmidt’s
subspace theorem due to Schlickewei~\cite{schlickewei1976linearformen},
proved that if $(u_n)_{n\ge 0}$ is a non-degenerate linear recurrence
sequence, then the greatest prime factor of $u_n$ tends to infinity, namely
\begin{equation}\label{ieq2}
P(u_n) \longrightarrow \infty \quad \text{as } n \longrightarrow \infty .
\end{equation}
Both results are ineffective in the sense that they do not yield explicit
bounds. In contrast, if the characteristic polynomial of $(u_n)_{n\geq 0}$ admits a
dominant root, that is,
\[
|\alpha_1| > |\alpha_j| \qquad (j=2,\ldots,q),
\]
then one obtains an effective lower bound of the form
\[
|u_n| > c_1 n^{\ell_1} |\alpha_1|^n
\]
for all $n > c_2$, where $c_1$ is equal to one half of the absolute value of
the coefficient of $x^{\ell_1}$ in the polynomial $f_1$, and $c_2$ is a
positive constant that can be computed effectively in terms of
$\alpha_1,\ldots,\alpha_q$ and $f_1,\ldots,f_q$.

In 1982, Stewart~\cite{stewart1982divisors} derived effective lower bounds
for both the greatest prime factor and the greatest square-free factor of
$u_n$ in the case where the characteristic polynomial admits a dominant
root. More precisely, assuming that $u_n \neq f_1(n)\alpha_1^n$, for any
$\varepsilon>0$ one has
\begin{equation}\label{ieq3}
P(u_n) > (1-\varepsilon)\log n
\end{equation}
and
\begin{equation}\label{ieq4}
Q(u_n) > n^{1-\varepsilon},
\end{equation}
for all $n>c_3$, where $c_3$ is a positive constant that can be computed
effectively in terms of $\varepsilon$ and the parameters
$\alpha_1,\ldots,\alpha_q$ and $f_1,\ldots,f_q$.
These bounds were obtained by applying a version of Baker’s theory on
linear forms in logarithms of algebraic numbers due to
Waldschmidt~\cite{waldschmidt1980lower}. Independently,
Shparlinski~\cite{shparlinski1980} established the estimate for
$P(u_n)$ in the special case where $f_1(n)$ is a non-zero constant, with
the exponent $1-\varepsilon$ replaced by a suitably small positive
constant. Later, both inequalities were further refined using a result
of Matveev~\cite{matveev2000explicit}, which gives explicit lower bounds
for linear forms in logarithms of algebraic numbers.

Let $ S=\{p_1,\ldots,p_s\} $ be a finite, non-empty set of distinct prime numbers.
For an integer $ m $, its $ S $-part, denoted by $ [m]_S $, is defined as
\[
[m]_S = p_1^{r_1}\cdots p_s^{r_s},
\]
where $ m = p_1^{r_1}\cdots p_s^{r_s} M $ with $ M $ coprime to each $ p_i $
for $ i=1,\ldots,s $, and $ r_1,\ldots,r_s $ are non-negative integers.
Using a $ p $-adic extension of Roth’s theorem due to Ridout~\cite{Ridout1958},
Mahler~\cite{mahler1966} showed that for certain classes of binary linear
recurrence sequences $ (u_n)_{n\ge 0} $, we have the following
\[
[u_n]_S < |u_n|^{\varepsilon}
\]
for all sufficiently large $ n $.
He also noted that this estimate implies that the greatest prime divisor
$ P(u_n) $ tends to infinity as $ n \to \infty $.
Subsequently, under different sets of hypotheses on the sequence
$ (u_n)_{n\ge 0} $, Bugeaud and Evertse~\cite{bugeaud2017parts} obtained effective,
though weaker, bounds for the $ S $-part of $ u_n $ of the shape
$ |u_n|^{1-c} $. For other related work, one can refer to (for example, \cite{bugeaud2017s, bugeaud2018digital, meher2022s}).

It is a classical consequence of the Skolem-Mahler-Lech theorem
\cite{everest2003recurrence} that the zero set of a linear recurrence sequence
$ (u_n)_{n\ge 0} $, namely the set  $ n\in\mathbb{N} $ for which
$ u_n=0 $, consists of the union of a finite set and finitely many arithmetic
progressions. In particular, although strong bounds-uniform in all parameters
and depending only on the order $ k $ are known for the total number of zeros
(see \cite{amoroso2011zeros}), no effective bound is available for the size of an
index $ n $ satisfying $ u_n=0 $.
A parametric variant of this problem, concerning families of linear recurrence sequences, was investigated by Ostafe and Shparlinski \cite{ostafe2022skolem}. For all but finitely many values of the parameter for
which the specialized sequence is non-degenerate, they established an explicit
upper bound for the largest possible zero. As a consequence, they showed that
the Skolem problem is effectively decidable outside a set of parameter values
$ \alpha\in\overline{\mathbb{Q}} $ of bounded height. More precisely, they
considered linear recurrence sequences of the form
\begin{equation}\label{ieq6}
U_n(X)=\sum_{i=1}^k f_i(X)\alpha_i(X)^n, \quad n\ge 0,
\end{equation}
where $ f_i(X),\alpha_i(X)\in \overline{\mathbb{Q}}(X) $. Furthermore, they analyzed other properties of the specialized values $ U_n(\zeta) $ when the
parameter is restricted to the set $ \mathbb{U} $ of all the roots of unity.

In this paper, we study the arithmetic properties of parametric linear recurrence sequences obtained by specializing the parameter at roots of unity.
Our first main result establishes strong lower bound for the largest prime
ideal divisor and for the radical of the principal ideal generated by the values
$ U_n(\zeta) $. More precisely, under suitable non-degeneracy and
non-exceptionality assumptions on the characteristic roots, we prove that for
all but finitely many roots of unity $ \zeta $, both the greatest prime ideal
divisor $ P(U_n(\zeta)) $ grow
at least logarithmically with $n$ and the norm of the radical $ Q(U_n(\zeta)) $ grow at least power of $n$, with effectively computable constants.
Our second result complements this result by providing an effective upper bound
for the norm of the $ S $-part of $ U_n(\zeta) $, showing that for
sufficiently large $ n $ it is strictly smaller than a fixed power of the
absolute norm of $ U_n(\zeta) $.

\section{Notation and Results} 
To formulate our results, we first recall some basic notions. A \emph{finite Blaschke product} is a rational function
$B(z) \in \mathbb{C}(z)$ of the form
\begin{equation}
\label{eq:2.1}
B(z)=\zeta \prod_{i=1}^{n}
\left(\frac{z-a_i}{1-\overline{a_i}z}\right)^{m_i},
\end{equation}
where $a_i$ are complex numbers lying in the open unit disc
$\{z \in \mathbb{C} : |z|<1\}$, the exponents $m_i$ are positive integers,
and $|\zeta|=1$.

A rational function $Q(z)$ of the form $Q(z)=B_1(z)/B_2(z)$, where
$B_1$ and $B_2$ are finite Blaschke products, is called a
\emph{quotient of finite Blaschke products}.
A pair of rational functions $(g_1(X), g_2(X)) \in \mathbb{C}(X)^2$
is said to be \emph{exceptional} if there exist quotients of finite
Blaschke products $Q_1, Q_2$ and a rational function $g$ such that
\[
g_1 = Q_1 \circ g
\quad \text{and} \quad
g_2 = Q_2 \circ g.
\]
Otherwise, the pair $(g_1(X), g_2(X))$ is called \emph{non-exceptional}.
We consider a linear recurrence sequence of the form \eqref{ieq6} and define
\begin{align*}
\mathcal{C}_{\alpha,f} =
\Bigl\{\zeta \in \overline{\mathbb{Q}} :
&\ \alpha_i(\zeta)/\alpha_j(\zeta) \text{ is a root of unity for some }
1 \le i < j \le k,\\
&\ \text{or } f_i(\zeta)=0 \text{ or } \alpha_i(\zeta)=0
\text{ for some } 1 \le i \le k \Bigr\}.
\end{align*}
Note that the elements of
$\mathcal{C}_{\alpha,f}$ have a bounded height (see \cite[Theorem~3.11]{silverman2007arithmetic}).
Hence, throughout this work, we restrict $\zeta$ to the set
$\mathbb{U}$ of all roots of unity. We denote by $\mathbb{Z}_{\overline{\mathbb{Q}}}$ the set of all algebraic
integers in $\overline{\mathbb{Q}}$.
For a polynomial $f \in \mathbb{C}[X]$, we write $\overline{f}$ for the
polynomial obtained by complex conjugating all coefficients of $f$.
Let $f_i(X), \alpha_i(X) \in \mathbb{Z}_{\overline{\mathbb{Q}}}[X]$,
$i=1,\ldots,k$, be nonzero polynomials of degree at most $d$.
We say that the linear recurrence sequence $(U_n(\zeta))_{n\geq 0}$,
defined as in \eqref{ieq6}, is of \emph{desired structure} if the following
conditions are satisfied:
\begin{itemize}
\item[(i)] For any $1 \le r < s < m \le k$, the pair of rational functions
$(\alpha_s/\alpha_r, \alpha_m/\alpha_r)$ is non-exceptional.
\item[(ii)] $U_n(\zeta)\neq f_1(\zeta)\alpha_1(\zeta)^n$ and
$|U_n(\zeta)|\geq 1$ for all $n\geq 0$.
\end{itemize}

Finally, we define
\[
P(U_n(\zeta)) := \max\{ N(\mathfrak{p}) : \mathfrak{p} \mid [U_n(\zeta)] \}
\quad \text{ and }\quad
Q(U_n(\zeta)) := \prod_{\mathfrak{p} \mid [U_n(\zeta)]} \mathfrak{p},
\]
where $[U_n(\zeta)]$ denotes the ideal generated by $U_n(\zeta)$ and $N(\mathfrak{p}):=\text{Norm} (\mathfrak{p})$ denotes the norm of the ideal $\mathfrak{p}$.

Our first main result is the following.

\begin{theorem}\label{thm1}
Let $(U_n(\zeta))_{n\geq 0}$ be a linear recurrence sequence of the desired structure. Let $K$ be a number field containing $\alpha_i(\zeta), f_i(\zeta)~(i=1,\ldots,k)$ and $D$ be its degree over $\mathbb{Q}$.
Then
\begin{equation}\label{neq12}
    P(U_n(\zeta))>C_2\log n\frac{\log \log n}{\log \log \log n}
\end{equation}
 and 
 \begin{equation}\label{neq13}
     N(Q(U_n(\zeta)))>n^{C_3(\log \log n)/\log \log \log n},
 \end{equation}
for all $n\geq C_1$ and for all but at most $2d(2dk^3/3 + 1)$ elements
$\zeta \in \mathbb{U} \setminus \mathcal{C}_{\alpha,f}$,
where $C_1, C_2$ and $C_3$ are computable constants depending on $\zeta, K$, $D$ the degree of the number field $K$, $r$ the rank of the unit group of $K$, $h_K$ the class number of the field $K$  and $k$.
\end{theorem}

Let $\mathcal{S}=\{\mathfrak{p}_1,\ldots,\mathfrak{p}_s\}$ be a finite non-empty set of distinct prime ideals. For a non-zero algebraic number $\alpha$ we write $[\alpha]=\mathfrak{p}_1^{r_1}\cdots\mathfrak{p}_s^{r_s}\mathfrak{a}$, where $r_1,\ldots,r_s$ are non-negative integers and $\mathfrak{a}$ is a non-zero ideal in the ring of integers of $K$ which is relatively prime to $\mathfrak{p}_1,\ldots,\mathfrak{p}_s$. We define the $S$-part $[\alpha]_S$ of $\alpha$ by
\[
[\alpha]_\mathcal{S}=\mathfrak{p}_1^{r_1}\cdots\mathfrak{p}_s^{r_s}.
\]
Then, we have the following result.
\begin{theorem}\label{thm2}
 Let $(U_n(\zeta))_{n\geq 0}$ be a linear recurrence sequence of the desired structure. Let $K$ be a number field containing $\alpha_i(\zeta), f_i(\zeta)~(i=1,\ldots,k)$  and $D$ be its degree over $\mathbb{Q}$, and $\mathcal{S}$ be a finite non-empty set of prime ideals in $\mathcal{O}_K$. Then there exist effectively computable positive numbers $C_4$ and $C_5$ depending on $U_n(\zeta), K, D, k$ and the cardinality $\#\mathcal{S}$ of $\mathcal{S}$ such that
 \[
N([U_n(\zeta)]_\mathcal{S})\leq |N(U_n(\zeta))|^{1-C_{4}}
\] for all $n\geq C_{5}$ and 
 for all but at most $2d(2dk^3/3 + 1)$  elements
$\zeta \in \mathbb{U} \setminus \mathcal{C}_{\alpha,f}$. 
   
\end{theorem}

The organization of this paper is as follows. In the next section, we present some technical ingredients needed to prove our results. In Section \ref{proofthm}, we present the proof of Theorems \ref{thm1} and \ref{thm2} following the methods of the papers \cite{ostafe2022skolem, stewart1982divisors, bugeaud2017parts} with suitable modifications. 

\section{Preliminaries}\label{sec3}
In this section, we give some results that are used to prove our main theorems. We shall define the height
$H(\beta)$ of an algebraic number $\beta$ by
\[
H(\beta) = |a_d| \prod_{i=1}^{d} \max\{1, |\beta_i|\},
\]
where
\[
a_d X^d + \cdots + a_0 = a_d \prod_{i=1}^{d} (X - \beta_i)
\]
is the minimal polynomial of $\beta$ in $\mathbb{Z}[X]$. The absolute logarithmic height of $\beta$ is defined by
\[
h(\beta)=\frac{1}{d}\left( \log |a_d|+\sum_{i=1}^d \log \max(1,|\beta_i| )\right).
\]

To prove our main results, we use the following lemma, which is a Baker-type result of Matveev \cite{matveev2000explicit}.

\begin{lemma}[\cite{matveev2000explicit}]\label{lem6}
    Denote by $\eta_1,\ldots,\eta_m$ algebraic numbers, not $0$ or $1$, by
$\log \eta_1,\ldots,\log \eta_m$ the principal values of their logarithms, by $D_1$ the degree of the number field $L=\Q(\eta_1,\ldots,\eta_m)$ over $\Q$, and by $b_1,\ldots,b_m$ rational integers. Define
\[
B=\max\{|b_1|,\ldots,|b_m|,2\}
\] and 
\[
A_i = \max\{D_1h(\eta_i), |\log \eta_i|, 0.16\} (1 \leq i \leq m),
\]
where $h(\eta)$ denotes the absolute logarithmic height of $\eta$. Consider the linear form
\[
\Lambda = b_1 \log \eta_1 + \cdots+ b_m \log \eta_m
\] 
and assume that $\Lambda\neq 0$. Then
\[
\log |\Lambda|\geq -C^mD_1^{m+2}A_1 \cdots A_m \log(D_1) \log(B).
\]

\end{lemma}

Furthermore, we need the following variant of Lemma \ref{lem6} in the proof of Theorem \ref{thm2}. Note that the factor $\log A_n$ in the denominator of the definition of $B$ plays an important role.

\begin{lemma}[\cite{matveev2000explicit} \cite{bugeaud2018linear}]\label{lem1}
Let $n \ge 2$ be an integer, let $\eta_1,\ldots,\eta_n$ be non-zero algebraic
numbers and let $b_1,\ldots,b_n$ be integers. Further, let $D_2$ be the degree over
$\mathbb{Q}$ of a number field containing the $\eta_i$, and let
$A_1,\ldots,A_n$ be real numbers with
\[
\log A_i \ge \max\left\{ h(\eta_i), \frac{|\log \eta_i|}{D_2}, \frac{0.16}{D_2} \right\},
\qquad 1 \le i \le n.
\]

Set
\[
B := \max\left\{ 1, \max_{1 \le j \le n} \left\{ |b_j| \frac{\log A_j}{\log A_n} \right\} \right\}.
\]

Then we have
\begin{align*}
    \log \left| \eta_1^{b_1} \cdots \eta_n^{b_n} - 1 \right|
> -4 \times 30^{n+4} (n+1)^{5.5}& D_2^{\,n+2}
\log(eD_2)\log(enB)\\
&\log A_1 \cdots \log A_n.
\end{align*}
\end{lemma}

\begin{lemma}\label{lem3}
Let $F$ be a number field of degree $D_3$, and let $\alpha \in \mathcal{O}_F \setminus \mathcal{O}_F^{\ast}$, where $\mathcal{O}_F$ and $\mathcal{O}_F^{\ast}$ denote the ring of integers of $F$ and the unit group of  $\mathcal{O}_F$ respectively. 
Then there is an effectively computable constant $C_{6}(F)$, depending on the fundamental units of 
$\mathcal{O}_F$, and an $\varepsilon \in \mathcal{O}_F^{\ast}$ such that
\[
|\overline{\varepsilon \alpha}| \le C_{6}(F)\, \bigl|N_{F/\mathbb{Q}}(\alpha)\bigr|^{1/D_3},
\]
where $|\overline{\alpha}|$ denotes the house of $\alpha$.
\end{lemma}

We recall that $|\overline{\alpha}|$, the house of $\alpha$, is defined as the maximum absolute value of the conjugates of $\alpha$ over $\mathbb{C}$.
\begin{proof}
See \cite[Lemma 1.3.8]{natarajan2020pillars}.
\end{proof}

\begin{lemma}\label{lem4}
Let $x,a \in \mathbb{R}$ . If
$\frac{a}{\log a} < x$,
then
\[
a < \max\{ e,\; 2x\log x \}.
\]
\end{lemma}

\begin{proof}
If  $\frac{a}{\log a} < x$, then 
\begin{equation}\label{eq15}
a < x \log a.
\end{equation}
Taking logarithms of the inequality \eqref{eq15}, we have $\log a - \log\log a < \log x.$
Further, we have that $\frac{\log a}{2} > \log\log a,$
which implies
\begin{equation}\label{eq16}
   \log a < 2\log x. 
\end{equation}
Combining \eqref{eq15} and \eqref{eq16}, we conclude that $a < 2x\log x.$
This completes the proof.
\end{proof}

\begin{lemma}\label{lem2}
Let $g_1(X), g_2(X) \in \mathbb{C}(X)$ be complex rational functions of degrees
$n_1$ and $n_2$, respectively. Then
\[
\#\{ z \in \mathbb{C} : |g_1(z)| = |g_2(z)| = 1 \} \le (n_1 + n_2)^2,
\]
unless $(g_1(X), g_2(X))$ is exceptional.
\end{lemma}
\begin{proof}
    See \cite[Theorem 2.2]{pakovich2020level}.
\end{proof}

We now proceed to the proofs of the main theorems. 

\section{Proof of Main Results}\label{proofthm}
\subsection{Proof of Theorem \ref{thm1}}
Let $\zeta \in \mathbb{U} \setminus \mathcal{C}_{\alpha,f}$ and $K$ be the field obtained by adjoining $\alpha_i(\zeta)$ and $f_i(\zeta)$ to $\Q$ for $i=1,\ldots,k$. Let $D$ be the degree of $K$ over $\Q$. Let $C_7,C_8,\ldots$ denote positive numbers which are effectively computable in terms of $\zeta, K, D, k, r$ and $h_K$.

First, consider the case where we have at least three dominant roots of $U_n(\zeta)$, that is, there exist distinct integers $1\leq t<s<m\leq k$ such that 
\[
|\alpha_t(\zeta)|=|\alpha_s(\zeta)|=|\alpha_m(\zeta)|
\] or equivalently,
\[
\frac{|\alpha_s(\zeta)|}{|\alpha_t(\zeta)|}=\frac{|\alpha_m(\zeta)|}{|\alpha_t(\zeta)|}=1.
\]
Since by hypothesis, $(\alpha_s/\alpha_t, \alpha_m/\alpha_t)$ is a non-exceptional rational function, from Lemma \ref{lem2} we see that for each $\binom{k}{3}$ possible choice of the triples $(t,s,m)$ there are at most $4d^2$ such $\zeta \in \mathbb{U} \setminus \mathcal{C}_{\alpha,f}$.
Hence, in total, we have excluded at most
\[
4d^2k(k-1)(k-2)/6=2d^2k(k-1)(k-2)/3
\] elements $\zeta\in \mathbb{U} \setminus \mathcal{C}_{\alpha,f} $. Next assume that we have exactly two dominant roots, that is, for some $1\leq i\neq j\leq k, |\alpha_i(\zeta)|=|\alpha_j(\zeta)|$. This implies  
\begin{equation}\label{eq7}
    \alpha_i(\zeta)\overline{\alpha_i}(\overline{\zeta})=\alpha_i(\zeta)\overline{ \alpha_i(\zeta)}=\alpha_j(\zeta)\overline{ \alpha_j(\zeta)}=\alpha_j(\zeta)\overline{\alpha_j}(\overline{\zeta}),
\end{equation}
 where $\overline{\zeta}$ is the complex conjugate of $\zeta$, which is again a root of unity. Now we use the fact that $\overline{\zeta}=\zeta^{-1}$ and thus \eqref{eq7} is reduced to a univariate polynomial equation in $\zeta$. Hence, there are at most $2d$ such elements in $\zeta\in \mathbb{U}\setminus\mathcal{C}_{\alpha,f}$ such that the sequence $(U_n(\zeta))_{n\geq 0}$ has two dominant roots.

Thus, for all but $2d(dk^3/3+1)$ elements $\zeta\in \mathbb{U} \setminus \mathcal{C}_{\alpha,f}$, the sequence $(U_n(\zeta))_{n\geq 0}$ has only one dominant root. Let for such an element $\zeta\in \mathbb{U}\setminus \mathcal{C}_{\alpha,f}$, we assume 
    \[    |\alpha_1(\zeta)|>\max_{i=2,\ldots,k}|\alpha_i(\zeta)|.
    \]
    Set 
    \[
    h_n(X)=f_2(X)\alpha_2(X)^n+\cdots+f_k(X)\alpha_k(X)^n
    \] and 
    \[
    R_n(\zeta)=\frac{U_n(\zeta)}{f_1(\zeta)\alpha_1(\zeta)^n}.
    \]
    Then from \eqref{ieq6},
    \[
    U_n(\zeta)=f_1(\zeta)\alpha_1(\zeta)^n+h_n(\zeta)
    \]
    and
    \begin{equation}\label{eq2}
         R_n(\zeta)=1+\frac{h_n(\zeta)}{f_1(\zeta)\alpha_1(\zeta)^n}.
    \end{equation}
We claim that there exists $\delta_j$ such that $0\leq \delta_j<1$ and $|\alpha_j(\zeta)|=|\alpha_1(\zeta)|^{\delta_j},~ 2\leq j\leq k$.
Consider the function $g(x)=|\alpha_1(\zeta)|^x, x\in \R$ which is a continuous increasing function whose range is $(0,\infty)$. Since $g(1)=|\alpha_1(\zeta)|$ and $0<|\alpha_j(\zeta)|<|\alpha_1(\zeta)|$, by the intermediate value theorem, there exists $\delta_j$ such that $0<\delta_j<1$ and $g(\delta_j)=|\alpha_j(\zeta)|$. That is, $|\alpha_j(\zeta)|=|\alpha_1(\zeta)|^{\delta_j}$. By choosing $\delta=\max_{2\leq j\leq k}\delta_j$, we infer that $|\alpha_j(\zeta)|\leq |\alpha_1(\zeta)|^{\delta} $ for all $2\leq j\leq k$. Then
\begin{align}\label{neq2}
\begin{split}|h_n(\zeta)|&=|f_2(\zeta)\alpha_2(\zeta)^n+\cdots+f_k(\zeta)\alpha_k(\zeta)^n|\\
    &\leq \sum_{j=2}^{k}|f_j(\zeta)| |\alpha_j(\zeta)|^n\leq C_{7} |\alpha_1(\zeta)|^{\delta n}
    \end{split}
    \end{align} 
for some  $0\leq \delta<1$. So, from \eqref{eq2} and \eqref{neq2},
    \begin{align}\label{eq9}
        |R_n(\zeta)-1|&=\left|\frac{h_n(\zeta)}{f_1(\zeta)\alpha_1(\zeta)^n}\right|<\frac{C_7|\alpha_1(\zeta)|^{\del n}}{|f_1(\zeta)||\alpha_1(\zeta)|^n}\\ \nonumber
        &<C_8 |\alpha_1(\zeta)|^{-n(1-\del)},
    \end{align}
 Next, we consider two cases as follows.
    
Case I: Suppose $|\alpha_1(\zeta)^{-n}f_1(\zeta)^{-1}U_n(\zeta)-1|\geq \frac{1}{2}$. Then $$2|h_n(\zeta)|>|f_1(\zeta)||\alpha_1(\zeta)|^n.$$
Using \eqref{neq2}, we get  
\[
|\alpha_1(\zeta)|^{n(1-\del)}<\frac{2C_7}{|f_1(\zeta)|}.
\] This implies $n<C_9$. But this case is not possible since we assume $n\geq C_1$. 

Case II: Suppose $|\alpha_1(\zeta)^{-n}f_1(\zeta)^{-1}U_n(\zeta)-1|< \frac{1}{2}$. Then there exists an integer $b_0$, with $|b_0|\leq 3B^{'}$, where $B^{'}=\max\{1,n\}=n$, such that 
\[
\Lambda:=|b_0\log(-1)-n\log\alpha_1(\zeta)-\log f_1(\zeta)+\log U_n(\zeta)|.
\]  
Using the fact that $|\log (1+x)|\leq 2|x|$ for $|x|\leq \frac{1}{2}$, we have
\begin{equation}\label{neq3}
    |\Lambda|\leq 2 |\alpha_1(\zeta)^{-n}f_1(\zeta)^{-1}U_n(\zeta)-1|.
\end{equation}
Suppose that the ideal generated by $U_n(\zeta)$ is written as,
$[U_n(\zeta)]=\mathfrak{p_1}^{a_1}\cdots\mathfrak{p_t}^{a_t}$ with $\mathfrak{p_1},\ldots,\mathfrak{p_t}$ are distinct prime ideals and $a_1,\ldots a_t$ are positive integers. Let $h_K$ denote the class number of the number field $K$, i.e.,  $\mathfrak{p}_i^{h_{K}}$ is principal ideal in $K$. Let $\mathfrak{p}_i^{h_K}=[\pi_i],$ ideal gnerated by $\pi_i\in K.$ Therefore,
    \[
    [U_n(\zeta)^{h_K}]=[\pi_i]^{a_1}\cdots[\pi_t]^{a_t}=[\pi_1^{a_1}\cdots\pi_t^{a_t}],
    \] which implies $U_n(\zeta)^{h_K}=\epsilon\pi_1^{a_1}\cdots\pi_t^{a_t}$, where $\epsilon$ is a unit in $K$. By Dirichlet unit theorem, we can write $ \epsilon=\mu \epsilon_1^{m_1}\cdots\epsilon_r^{m_r}$, where $\mu$ is a root of unity, $\epsilon_1,\ldots, \epsilon_r$ are fixed fundamental units, $m_1, \ldots, m_r$ are integers and $r$ is the rank of the unit group in $K$. Hence, we write
    \[
    U_n(\zeta)^{h_K}=\mu \epsilon_1^{m_1}\cdots\epsilon_r^{m_r}\pi_1^{a_1}\cdots\pi_t^{a_t}
    \]
    Taking logarithm on both sides, we get
\[
\log(U_n(\zeta))=\frac{1}{h_K}\left(\log \mu+m_1\log\epsilon_1+\cdots+m_r\log \epsilon_r+a_1\log \pi_1+\cdots+a_t\log\pi_t \right).
\] Thus,
\begin{align*}
    \Lambda=&|b_0\log(-1)-n\log\alpha_1(\zeta)-\log f_1(\zeta)+\\
    &\frac{1}{h_{K}}\left(\log \mu+m_1\log\epsilon_1+\cdots+m_r\log \epsilon_r+a_1\log \pi_1+\cdots+a_t\log\pi_t \right)|
\end{align*}
Now we apply Lemma \ref{lem6} with 
$\eta_0=-1, \eta_1=\alpha_1(\zeta), \eta_2=f_1(\zeta), \eta_3=\mu,\eta_4=\epsilon_1,\ldots,\eta_{r+4}=\epsilon_r, \eta_{r+5}=\pi_1, \ldots, \eta_{r+t+5}=\pi_t$.
As shown in the proof of inequality (4.3.2) in~\cite{evertse2022effective}, it is shown that
\[
\max \left\{
3, |m_1|,\ \ldots,\ |m_r|
\right\}
\leq C_8\, h(\epsilon),
\] where $C_8$ is a constant depending on the field $K$. Now
\begin{equation}\label{neq6}
    h(\epsilon)=h\left(\frac{U_n(\zeta)^{h_K}}{\prod_{i=1}^t\pi_i^{a_i}}\right)\leq h_K\cdot h(U_n(\zeta))+\sum_{i=1}^ta_ih(\pi_i).
\end{equation}
We have
\begin{align*}
    h(U_n(\zeta))&=h(f_1(\zeta)\alpha_1(\zeta)^n+\cdots+f_k(\zeta)\alpha_k(\zeta)^n)\\
    &\leq h(f_1(\zeta))+nh(\alpha_1(\zeta))+\cdots h(f_k(\zeta))+nh(\alpha_k(\zeta))\leq C_9n,
\end{align*}
where $C_9$ is a constant depending on $k$ and $\zeta$. Also since $|N(U_n(\zeta))|=\prod_{i=1}^tN(\mathfrak{p}_i)^{a_i}$, we have $N(\mathfrak{p}_i)^{a_i}\leq |N(U_n(\zeta))| $. So,
\[
a_i\leq \frac{\log |N(U_n(\zeta))|}{\log N(\mathfrak{p}_i)}\leq \frac{\log |N(U_n(\zeta))|}{\log 2},
\] since $N(\mathfrak{p}_i)\geq 2$.
Also note that
\begin{align*}
    |N(U_n(\zeta))|=\prod_{j=1}^{D}|\sigma_j(U_n(\zeta))|&=\prod_{j=1}^{D}|\sigma_j(f_1(\zeta)\alpha_1(\zeta)^n+\cdots+f_k(\zeta)\alpha_k(\zeta)^n)|\\
    &\leq \prod_{j=1}^{D} C_{\sigma_j}M_{\sigma_j}^n,
\end{align*} 
where $\sigma_j(j=1,\ldots, D)$ are the isomorphisms in Gal($K/\Q$) with $M_{\sigma_j}=\max_{1\leq i\leq k}\{|\sigma_j(\alpha_i(\zeta))|\}$ and  $C_{\sigma_j}=\max_{1\leq i\leq k}\{|\sigma_j(f_i(\zeta))|\}$. Therefore, 
$a_i\leq C_{10}n,~ (i=1,\ldots,t), C_{10}$ is a constant depending on $D, K, k$ and $\zeta$. Now we find the upper bound for $h(\pi_i)$ for $i=1,\ldots,t$.
Let us denote the house of $\pi_i$ by $|\overline{\pi_i}|=\max_{j=1,\ldots,D}\{|\sigma_j(\pi_i)|\}$. Thus,
\[
h(\pi_i)=\frac{1}{D}\sum_{j=1}^D\max\left(\log |\sigma_j(\pi_i)|,0 \right)\leq \log |\overline{\pi_i}|\leq C_{11}\log |N(\pi_i)|,
\]  where the last inequality can be obtained from Lemma \ref{lem3}. Since $|N(\pi_i)|=C_{12}N(\mathfrak{p}_i)$ and $
N(\mathfrak{p}_i)\leq |N(U_n(\zeta))|$, we have 
$\log |N(\pi_i)|\leq \log |N(U_n(\zeta))|\leq C_{13}n$. Therefore $h(\pi_i)\leq C_{14}n$ and \eqref{neq6} becomes
\[
h(\epsilon)\leq h_KC_9n+tC_{15}n^2\leq C_{16}tn^2.
\]
Thus, 
\[
B=\max\left\{ b_0,n,\frac{m_1}{h},\ldots,\frac{m_r}{h},\frac{a_1}{h},\ldots, \frac{a_t}{h}\right\}\leq C_{17}tn^2.
\] Furthermore, $|\log (-1)|=\pi, h(-1)=0,~ \log A_0=\frac{\pi}{D},~ \log A_1, \ldots, \log A_{r+4}$ are constants depending on the field $K,\zeta, r$ and $A_j=\max\{Dh(\pi_i), |\log \pi_i|, 0.16\}$ for $j=r+5,\ldots,r+t+5, ~ i=1,\ldots,t$.
Since $h(\pi_i)\leq C_{11}\log |N(\pi_i)|$ and that $|\log \pi_i|\leq C_{18}h(\pi_i)$, we choose $A_j=C_{19}\log |N(\pi_i)|$.  
    Using Lemma \ref{lem6} and \eqref{neq3}, we get
    \begin{equation}\label{eq10}
        |R_n(\zeta)-1|>\frac{1}{2}\exp\left(-C_{20}^t\log |N(\pi_1)|\cdots\log|N(\pi_t)| \log (nt) \right).
    \end{equation}
Comparing \eqref{eq9} and \eqref{eq10}, we get
 \[
    C_8|\alpha_1(\zeta)|^{-n(1-\del)}>\frac{1}{2} \exp\left(- C_{20}^t\log |N(\pi_1)|\cdots\log|N(\pi_t)| \log (nt)\right).
    \] This implies
    \[
    {-n(1-\del)} \log |\alpha_1(\zeta)|>-C_{20}^t\log |N(\pi_1)|\cdots\log|N(\pi_t)| \log (nt)-\log C_8
    \] and hence
    \begin{equation}\label{eq11}
        C_{21}\frac{n}{\log (nt)} <  C_{22}^t\log |N(\pi_1)|\cdots\log|N(\pi_t)|.
    \end{equation} 
   By the arithmetic-geometric mean inequality
   \begin{equation}\label{neq8}
        \prod_{i=1}^t\log |N(\pi_i)|\leq \left( \frac{\log\left(\prod_{i=1}^t |N(\pi_i)|\right)}{t}\right )^t
   \end{equation}
   Since $\prod_{i=1}^t |N(\pi_i)|=\prod_{i=1}^tN(\mathfrak{p}_i)^{h_K}=N(Q(U_n(\zeta)))^{h_K}$ it follows from \eqref{neq8} and \eqref{eq11} that
   \[
   C_{21} \frac{n}{\log (nt)}<C_{22}^t\left(\frac{h_K\log (N(Q(U_n(\zeta))))}{t}\right)^t.
   \] That is,
   \[
   \log n-\log (\log n+ \log t)+\log C_{21}<t\log C_{22}+t\log \left(\frac{h_K \log(N(Q(U_n(\zeta))))}{t}\right).
   \] Since $\log (\log n+\log t)=\log\left(\log n\left(1+\frac{\log t}{\log n}\right)\right)= \log \log n+\log \left(1+\frac{\log t}{\log n}\right)$ and using the fact that $\frac{\log t}{\log n}\leq \frac{\log t}{\log 2}$ and hence $\log \left(1+\frac{\log t}{\log n}\right)\leq \log \log t\leq C_{23}t$, we get
   \[
   \log n-\log \log n+\log C_{21}<tC_{24}+t\log \left(\frac{h_K \log(N(Q(U_n(\zeta))))}{t}\right),
   \]
    where $C_{24}=\log C_{22}+C_{23}$.
  So for sufficiently large $n$,
  \[
  \frac{\log n}{t}-C_{25}\frac{\log \log n}{t}<C_{24}+\log \left(\frac{h_K \log(N(Q(U_n(\zeta))))}{t}\right),
  \]
  hence
 \begin{equation}\label{neq9}
     C_{26}t\exp\left(\frac{\log n}{t}-C_{25}\frac{\log \log n}{t}\right)<\log(N(Q(U_n(\zeta))))
 \end{equation}
 We assume that $t$ is less than $\log n/\log \log \log n$. Put
 \[
 F(t)=t\exp\left(\frac{\log n}{t}-C_{25}\frac{\log \log n}{t}\right)
 \] and notice that $F$ is decreasing for $t$ in the range from 1 to $\log n-C_{25}\log \log n$. Thus for $n$ sufficiently large $\log n/\log \log \log n$ is less than $\log n-C_{25} \log \log n$. Therefore
\[
F(t)\geq F\left(\frac{\log n}{\log\log\log n}\right).
\] 
Substituting this into \eqref{neq9}, we obtain
\[
\log (N(Q(u(n))))
>
C_{26}\,
\frac{\log n}{\log\log\log n}
\exp\left(
\frac{\log n-C_{25}\log\log n}
{\log n/\log\log\log n}
\right).
\]
Now simplifying the exponent we get

\[
\frac{\log n-C_{25}\log\log n}
{\log n/\log\log\log n}
=
\log\log\log n
-
C_{25}
\frac{(\log\log n)(\log\log\log n)}
{\log n}.
\]
Since
\[
C_{25}
\frac{(\log\log n)(\log\log\log n)}
{\log n}
\longrightarrow 0
\qquad (n\to\infty),
\]
 it follows that for sufficiently large $n$,
 \begin{equation}\label{neq10}
     \exp(C_{27}(\log n \log \log n)/\log \log \log n)<N(Q(U_n(\zeta))),
 \end{equation} as required. On the other hand if $t$ is at least $\log n/\log \log \log n$ then the product of the first $t$ primes exceeds $\exp((\log n \log \log n)/2 \log \log \log n)$ for $n$ sufficiently large and therefore
 \begin{equation}\label{neq11}
     \exp((\log n \log \log n)/2 \log \log \log n)< N(Q(U_n(\zeta))).
 \end{equation}
Thus \eqref{neq13} follows from \eqref{neq10} and \eqref{neq11}.

For any $n$,
\begin{equation}\label{neq14}
    N(Q(U_n(\zeta)))\leq \prod_{N(\mathfrak{q})\leq P(U_n(\zeta))}N(\mathfrak{q})<\exp(C_{28}P(U_n(\zeta)))
\end{equation} and thus\eqref{neq12} follows from \eqref{neq13} and \eqref{neq14}.
 This completes the proof of Theorem \ref{thm1}. \qed
\subsection{Proof of Theorem \ref{thm2}}
Let $\zeta \in \mathbb{U} \setminus \mathcal{C}_{\alpha,f}$ and $K$ be the field obtained by adjoining $\alpha_i(\zeta)$ and $f_i(\zeta)$ to $\Q$ for $i=1,\ldots,k$. Let $D$ be the degree of $K$ over $\Q$ and $\mathcal{S}=\{\mathfrak{p}_1,\ldots,\mathfrak{p}_s\}$ be a finite non-empty set of distinct prime ideals. Let $C_{29},C_{30},\ldots$ denote positive numbers that are effectively computable in terms of $U_n(\zeta), K, D, k$ and $\#\mathcal{S}$.
    As we proceeded with the proof of Theorem \ref{thm1}, we find that for all but $2d(dk^3/3+1)$ elements $\zeta\in \mathbb{U} \setminus \mathcal{C}_{\alpha,f}$, the sequence $(U_n(\zeta))_{n\geq 0}$ has only one dominant root. For such an element $\zeta\in \mathbb{U}\setminus \mathcal{C}_{\alpha,f}$,
    \[    |\alpha_1(\zeta)|>\max_{i=2,\ldots,k}|\alpha_i(\zeta)|
    \] and 
    \[
    |U_n(\zeta)-f_1(\zeta)\alpha_1(\zeta)^n|=|h_n(\zeta)|\leq |\alpha_1(\zeta)|^{\delta n}, \quad \text{for some $0\leq \delta<1$}.
    \] Thus we have 
    \begin{equation}\label{eq8}
        | R_n(\zeta)-1|\leq C_8|\alpha_1(\zeta)|^{-n(1-\del)},
    \end{equation} where $R_n(\zeta)$ is defined as in \eqref{eq2}.
    Now let 
    \[
     [U_n(\zeta)]=\mathfrak{p}_1^{r_1}\cdots\mathfrak{p}_s^{r_s}\mathfrak{a},
    \] where $r_1,\ldots,r_s$ are non-negative integers and $\mathfrak{a}$ is a non-zero ideal relatively prime to $\mathfrak{p}_1,\ldots,\mathfrak{p}_s$.
    Since $R_n(\zeta)\neq 1$, we apply Lemma \ref{lem1} with $\eta_1=f_1(\zeta), \eta_2=\alpha_1(\zeta)$ and $\eta_3=U_n(\zeta)$ and $b_1=-1,b_2=-n,b_3=1$. Also, $\log A_1, \log A_2$ are constants depending on $\zeta$ and $K$, and 
\[
\log A_3 \ge \max\left\{ Dh(U_n(\zeta)), |\log U_n(\zeta)|, 0.16 \right\}.
\]
Also,
\[
h(U_n(\zeta))=\frac{1}{D}\sum_{i=1}^D\max\left(\log |\sigma_i(U_n(\zeta))|,0 \right),
\] where $\sigma_i~(i=1,\ldots, D)$ are the isomorphisms in Gal($K/\Q$).
Let us assume that $\sigma_1(U_n(\zeta))=U_n(\zeta)$ , where $\sigma_1$ is the identity isomorphism, and denote the house of $U_n(\zeta)$ by
\[
|\overline{U_n(\zeta)}|=\max_{i=1,\ldots,D}\left\{|\sigma_i(U_n(\zeta))|\right\} \geq |U_n(\zeta)|\geq 1.
\]
Then
\[
\max (\log |\sigma_i(U_n(\zeta))|,0)\leq \log|\overline{U_n(\zeta)}|, \quad \text{for all } i=1,\ldots, D.
\]
That is, 
\begin{equation}\label{eq13}
    h(U_n(\zeta))\leq \log|\overline{U_n(\zeta)}|\leq C_{29} \log |N(U_n(\zeta))|,
\end{equation} where the last inequality can be obtained from Lemma \ref{lem3}. Since $|U_n(\zeta)|\geq 1$,
\begin{equation}\label{eq14}
\begin{split}
\log |U_n(\zeta)|=\max (\log |U_n(\zeta)|,0)&\leq \sum_{i=1}^D\max (\log \sigma_i|U_n(\zeta)|,0)\\
&=D\cdot h(U_n(\zeta)).
    \end{split}
\end{equation}
Thus, from \eqref{eq13} and \eqref{eq14} we can choose $A_3=C_{30}|N(U_n(\zeta))|$. Also
    \[
    B=\max\left\{1, \max\left\{ \frac{\log A_1}{\log A_3},n\cdot \frac{\log A_2}{\log A_3}, \frac{\log A_3}{\log A_3}\right\}\right\}=C_{31} \frac{n}{\log A_3}.
    \] 
    Using Lemma \ref{lem1}, we get
    \begin{equation}\label{eq12}
        |R_n(\zeta)-1|>\exp\left(-C_{32} \log A_3 \log \frac{n}{\log A_3} \right).
    \end{equation}
    Comparing  \eqref{eq8} and \eqref{eq12}, we get
    \[
    n(1-\delta)C_{33}\log |\alpha_1(\zeta)|<C_{32} \log A_3\log\frac{n}{\log A_3},
    \]
    and hence
    \[
    n<C_{34}\log A_3\log \frac{n}{\log A_3}.
    \]    
    By Lemma \ref{lem4}, we have $n<C_{35}\log A_3.$
    Since $A_3=\prod_{i=1}^sN(\mathfrak{p}_i)^{r_i} N(\mathfrak{a}),$
    we get
    \begin{equation}\label{neq1}        n<C_{35}\left(\log\left(\prod_{i=1}^sN(\mathfrak{p}_i)^{r_i} \right)+\log (N(\mathfrak{a}))\right).
    \end{equation}
     We distinguish into two cases.
    
    Case I: If $\Big(N(\mathfrak{a})>\prod_{i=1}^s N(\mathfrak{p}_i)^{r_i}\Big)$, then from \eqref{neq1} we have  $ n<2C_{35}\log (N(\mathfrak{a})).$  That is, 
    \[
    N(\mathfrak{a})>\exp\left( C_{36}n\right).
    \] Since
   \begin{align*}
    |N(U_n(\zeta))|=\prod_{j=1}^{D}|\sigma_j(U_n(\zeta))|&=\prod_{j=1}^{D}|\sigma_j(f_1(\zeta)\alpha_1(\zeta)^n+\cdots+f_k(\zeta)\alpha_k(\zeta)^n)|\\
    &\leq \prod_{j=1}^{D} C_{\sigma_j}M_{\sigma_j}^n,
\end{align*} 
where $\sigma_j(j=1,\ldots, D)$ are the isomorphisms in Gal($K/\Q$) with $M_{\sigma_j}=\max_{1\leq i\leq k}\{|\sigma_j(\alpha_i(\zeta))|\}$ and  $C_{\sigma_j}=\max_{1\leq i\leq k}\{|\sigma_j(f_i(\zeta))|\}$ and we have 
   \[
   n\geq C_{37}\log |N(U_n(\zeta))|.
   \]   
Hence,
    \begin{align*}
        \frac{N([U_n(\zeta)])}{N([U_n(\zeta)]_S)}&=N(\mathfrak{a})> \exp\left( C_{36}n\right)\\
        &>\exp(C_{36}C_{37}\log |N(U_n(\zeta))|)\geq |N(U_n(\zeta))|^{C_{38}}.
    \end{align*}
    Therefore,
    \[
    N([U_n(\zeta)]_S)\leq |N(U_n(\zeta))|^{1-C_{38}}.
    \]
    Case II: If $\Big(N(\mathfrak{a})\leq \prod_{i=1}^s N(\mathfrak{p}_i)^{r_i}\Big)$, then \eqref{neq1} becomes,
    \[
    n<2C_{35}\log\left(\prod_{i=1}^sN(\mathfrak{p}_i)^{r_i}\right). 
    \]
This completes the proof of Theorem \ref{thm2}. 
\qed

{\bf Acknowledgment:} We are very grateful to Prof. A. Ostafe and Prof. I.E. Shparlinski for
their useful comments in the preliminary version of this paper. D.N. and S.S.R. are supported by a grant from Anusandhan National Research Foundation (File No.:CRG/2022/000268).

\end{document}